\documentclass{amsart}
\usepackage{amscd}
\usepackage{amssymb}

\def\fx{\frac{\partial f}{\partial x}}
\def\fy{\frac{\partial f}{\partial y}}
\def\fz{\frac{\partial f}{\partial z}}

\def\uOmega{\underline{\Omega}}
\def\bu{\bullet}
\def\coker{\operatorname{coker}}
\def\cone{\operatorname{cone}}
\def\ker{\operatorname{ker}}
\def\nil{\operatorname{nil}}
\def\oo{\otimes}
\def\sing{\mathrm{sing}}
\def\Ext{\operatorname{Ext}}
\def\Hom{\operatorname{Hom}}
\def\len{\operatorname{length}}
\def\Tot{\operatorname{Tot}}
\def\Proj{\operatorname{Proj}}
\def\Spec{\operatorname{Spec}}

\def\tors{\operatorname{tors}}
\def\zar{\mathrm{zar}}
\def\lra{\longrightarrow}
\newcommand{\mathdot}{{\mathbf{\scriptscriptstyle\bullet}}}
\newcommand{\Hcdh}{H_\cdh}
\def\red{\mathrm{red}}

\newcommand{\comment}[1]{}  
\def \ie{{\it i.e.,\ }}

\def\cF{\mathcal F}
\def\cG{\mathcal G}

\def\ctL{\mathcal{L}}

\def\cO{\mathcal O}

\def\cK{\mathcal K}

\newcommand{\C}{\mathbb{C}}
\newcommand{\bL}{\mathbb{L}} 
\newcommand{\Q}{\mathbb{Q}}
\newcommand{\R}{\mathbb{R}}
\newcommand{\Z}{\mathbb{Z}}
\newcommand{\N}{\mathbb{N}}

\newcommand{\bbH}{\mathbb H}

\def\cdh{\mathrm{cdh}}

\def\bu{\bullet}

\def\nil{\operatorname{nil}}
\def\Pic{\operatorname{Pic}}
\def\Sing{\operatorname{Sing}}
\def\map#1{{\buildrel #1 \over \lra}}
\def\lmap#1{{\buildrel #1 \over \longleftarrow}}

\newcommand{\SchF}{\mathrm{Sch}/F}

\input xypic
\xyoption{all} 

\numberwithin{equation}{section}

\theoremstyle{plain}
\newtheorem{thm}[equation]{Theorem}

\newtheorem{lem}[equation]{Lemma}
\newtheorem{prop}[equation]{Proposition}

\theoremstyle{definition}
\newtheorem{defn}[equation]{Definition}

\theoremstyle{remark}
\newtheorem{rem}[equation]{Remark}
\newtheorem{ex}[equation]{Example}
\newtheorem{subremark}{Remark}[equation] 
\newtheorem{subex}[subremark]{Example} 

\begin{document}

\title{A negative answer to a question of Bass}

\author{G. Corti\~nas}
\thanks{Corti\~nas' research was supported by Conicet 
and partially supported by grants PICT 2006-00836, UBACyT X051, 
PIP 112-200801-00900, and MTM2007-64704 (Feder funds).}
\address{Dep. Matem\'atica, FCEyN-UBA\\ Ciudad Universitaria Pab 1\\
1428 Buenos Aires, Argentina}
\email{gcorti@gm.uba.ar}

\author{C. Haesemeyer}
\thanks{Haesemeyer's research was partially supported by NSF grant DMS-0652860}
\address{Dept.\ of Mathematics, University of California, Los Angeles CA
90095, USA}
\email{chh@math.ucla.edu}

\author{Mark E. Walker}
\thanks{Walker's research was partially supported by NSF grant DMS-0601666.}
\address{Dept.\ of Mathematics, University of Nebraska - Lincoln,
  Lincoln, NE 68588, USA}
\email{mwalker5@math.unl.edu}

\author{C. Weibel}
\thanks{Weibel's research was
supported by NSA grant MSPF-04G-184 and the Oswald Veblen Fund.}
\address{Dept.\ of Mathematics, Rutgers University, New Brunswick,
NJ 08901, USA} \email{weibel@math.rutgers.edu}

\date{\today}

\begin{abstract}
We address Bass' question, on whether $K_n(R)=K_n(R[t])$
implies $K_n(R)=K_n(R[t_1,t_2])$. In a companion paper,  
we establish a positive answer to this question when $R$ is 
of finite type over a field of infinite transcendence degree over the
rationals. Here we provide an example of an 
isolated surface singularity over a number field  for
which the answer the Bass' question is ``no'' when $n=0$.
\end{abstract}
\maketitle

\section*{Introduction}

In 1972, H. Bass posed the following question
(see \cite{Bass73}, question (VI)$_n$):
\begin{center}
Does $K_n(R)=K_n(R[t])$ imply that $K_n(R)=K_n(R[t_1,t_2])$?
\end{center}
Bass' question was inspired by
Traverso's theorem \cite{Traverso}, from which it follows that $\Pic(R)
= \Pic(R[t])$ implies $\Pic(R) = \Pic(R[t_1,t_2])$. 

In the companion paper \cite{nk}, we show that the answer to 
Bass' question is ``yes'' for rings
of finite type over fields having infinite
transcendence degree over $\Q$. In this paper, we give an
example showing the
answer is ``no'' in general, even when $n=0$. That is, there is a ring
$R$ for which every finitely generated projective module over $R[t]$
is the extension, up to stable isomorphism, of a projective module
over $R$, but not every finitely generated projective module over
$R[t_1, t_2]$ is so extended.

Our example is the isolated hypersurface singularity 
$$
R = F[x,y,z]/(z^2 + y^3 + x^{10} + x^7 y),
$$
where $F$ is any algebraic field extension of $\Q$. 
(The proof is given in Theorem \ref{thm:dBcompute}.)
This example was first studied by J.~ Wahl \cite{Wahl}. 

Our proof that $R$ indeed gives a negative answer to Bass' question
uses what we call {\em generalized du Bois invariants}, $b^{p,q}$,
of an isolated singularity in characteristic zero; see \eqref{eq:dBinvars}.
The (ordinary) du Bois invariants were  introduced
by Steenbrink \cite{Steenbrink} using the du Bois complexes
$\uOmega^p$, $p \geq 0$. They can equivalently be defined using 
sheaf cohomology in Voevodsky's $cdh$ topology thanks to the 
natural isomorphism (see Lemma \ref{lem:duBois=cdh})
$$
\bbH^*_\zar(X, \uOmega^p_X) \cong H^*_\cdh(X, \Omega^p).
$$
The generalized du Bois invariants are defined as the
cohomology of the complex obtained by patching together the du Bois
complexes $\uOmega^p$  and the higher cotangent complexes used to 
define Andr\'e-Quillen homology.  
The Euler characteristics of these patched together complexes, 
written $\chi^p$ for $p \geq 0$, turn out to be constant in suitably nice
families (see Theorem \ref{thm:chi}). In particular, we prove in 
Proposition \ref{prop:Wahl} that $\chi^p(R_a)$ is independent of $a\in F$
where $R_a = F[x,y,z]/(z^2 + y^3 + x^{10} + ax^7 y)$. Since the ring $R_0$ 
is graded, the values of $\chi^p(R_0) = \chi^p(R_1)$ are easy to compute,
and these computations allow us to prove our assertion about $R = R_1$.

\subsection*{Notation} 
Throughout this paper, $F$ denotes a field of characteristic zero.
By ``a scheme over $F$'' we mean a separated scheme  
of finite type over $F$. 
We write $\SchF$ for the category of all
such schemes. Unless otherwise stated, Hochschild homology and modules of
K\"ahler differentials will be taken relative to $F$. That is, we write
$\Omega^q_X$ and $HH_q(X)$ for $\Omega^q_{X/F}$ and $HH_q(X/F)$.

\section{On $cdh$-cohomology and nil $K$-theory}

For any functor $G$ from rings to an abelian category,
$NG$ is the functor with $NG(R)$ defined to be the kernel of the map $G(R[t])\to G(R)$ induced by
evaluation at $t=0$. Since $G(R[t])\to G(R)$ is
split by the canonical map $G(R) \to G(R[t])$, the functor $NG$ is a summand
of the functor $R \mapsto  G(R[t])$.
We define $N^2G=N(NG)$.

It is convenient to phrase Bass' question in terms of Bass' Nil groups,
$NK_*(R)$, as follows:
\begin{center}
Does $NK_n(R)=0$ imply that $N^2K_n(R)=0$?
\end{center}

Our example uses the following theorem from the companion paper \cite{nk}.
(The notation in this theorem is discussed below; 
the particular forms of $V$ and $W$ in this theorem reflect extra 
structure not relevant for this paper.)

\begin{thm} \cite[Theorems 0.1 and 0.7]{nk} \label{mainnk}
Let $R$ be a normal domain of dimension $2$ that is
of finite type over $\Q$. Then, letting $V$ and $W$ denote the
countably-infinite dimensional $\Q$
vector spaces $t \Q[t]$ and $\Omega^1_{\Q[t]}$, we have:
\begin{enumerate}
\item[a)] $NK_0(R) \cong H^1_{\cdh}(R,\Omega^1)\otimes_\Q V$.
\item[b)] $NK_{-1}(R) \cong H^1_{\cdh}(R,\cO)\otimes_\Q V$.
\item[c)] If $NK_0(R) = 0$, then  $K_0(R[t_1, t_2]) \cong K_0(R)
  \oplus \left( NK_{-1}(R) \otimes_\Q W \right)$.
\end{enumerate}
\end{thm}

In particular, 
for $R$ as in Theorem \ref{mainnk},
 the answer to
Bass' question with $n=0$ 
is ``no'' if and only if
$H^1_{\cdh}(R,\Omega^1) = 0$ and 
$H^1_{\cdh}(R,\cO) \ne 0$.

Recall that the $cdh$ topology on $\SchF$, written $(\SchF)_\cdh$, 
is the Grothendieck topology generated by Nisnevich
open covers and abstract blow-up squares \cite{SVBK}. 
If $\cG$ is a presheaf on $\SchF$,
by $H^*_{\cdh}(X, \cG)$, we mean the $cdh$-sheaf cohomology of the
$cdh$-sheafification $\cG$. For example, $H^*_{\cdh}(X, \Omega^p)$, for
$p \geq 0$, refers to the $cdh$-cohomology of the 
$cdh$-sheafification of $Y \mapsto \Omega^p_Y$. (Of course, 
$\Omega^0_{Y} = \cO_Y$.) When $X = \Spec R$ for an $F$-algebra $R$
of finite type over $F$, we usually
write $H^*_{\cdh}(R, \cG)$ for $H^*_{\cdh}(\Spec R, \cG)$. 

\section{Generalized du Bois invariants, $\chi^p$ and deformations}
\label{sec:Equiresolution}

In this section, we construct invariants of isolated singularities,
called the {\em generalized du Bois invariants} $b^{p,q} \in \N$,
which for $q>0$ coincide with the du Bois invariants introduced by
Steenbrink \cite{Steenbrink}. For isolated singularities that are also
local complete intersections, for each fixed $p$ only a finite number
of the integers $b^{p,q}$ are nonzero. Thus it makes sense to define
$\chi^p := \sum_q (-1)^q b^{p,q}$ in this situation. The main result
of this section is Theorem \ref{thm:chi}, that the $\chi^p$ are
invariant under suitably nice deformations.

Recall that we work over a field $F$ of characteristic zero.
Several of the results we quote from here on --- in
particular anything involving du Bois complexes --- have been proved
under the assumption that  $F = \C$; however, flat base
change implies that they all remain valid over an arbitrary field $F$ of
characteristic $0$.

Fix a scheme $X$ of finite type over $F$ and 
choose a proper simplicial hyperresolution 
$\pi:Y_\mathdot\to X$. Following \cite{duBois} we fix $p$ and
we consider the $p$-th {\em du Bois complex}
\[
\uOmega^p_X=\R\pi_*\Omega^p_{Y_\mathdot}.
\]
Du Bois shows in \cite{duBois} that the assignment $X\mapsto \uOmega^p_X$
is natural in $X$ up to unique isomorphism in the derived category.
The relevance for us lies in the fact that the Zariski hypercohomology 
of the complex $\uOmega^p_X$ computes $H^*_\cdh(X,\Omega^p)$:

\begin{lem}\label{lem:duBois=cdh}
Let $X$ be a scheme of finite type over $F$ and $p\ge 0$. 
Then there is a natural isomorphism
$$
\bbH^*_\zar(X, \uOmega^p_X) \cong H^*_\cdh(X, \Omega^p).
$$
\end{lem}

A very similar observation for the $h$-topology has been made by
Ben Lee \cite{BenLee}; the proof we give here is based upon
the proof of \cite[4.1]{toric}.

\begin{proof} 
Recall that $H^*_\cdh(X,\Omega^p)$ is
the Zariski hypercohomology of the complex $\R a_*a^*\Omega^p|_X$,
where $a: (\SchF)_\cdh \to (\SchF)_\zar$ is the morphism of sites
and $|_X$ denotes the restriction from the big Zariski site
$(\SchF)_\zar$ to $X_\zar$. Let $Y_\bu \to X$ be a simplicial
hyperresolution.
By \cite[2.5]{chw}, we have a quasi-isomorphism on $X_\zar$
$$
\Omega^p_{Y_n} \map{\simeq} \R a_*a^*\Omega^p|_{Y_n}
$$
since each $Y_n$ is smooth.
Using also  \cite[4.3]{toric},
we have a diagram of equivalences
\[
\R a_*a^*\Omega^p|_X ~\map{\simeq}~ \R\pi_*(\R a_*a^*\Omega^p|_{Y_\mathdot})
\lmap{\simeq} \R\pi_*\Omega^p_{Y_\mathdot}=\uOmega^p_X.
\]
Applying $\bbH^*_\zar(X,-)$ yields
$H^*_\cdh(X,\Omega^p) \cong \bbH^*_\zar(X, \uOmega^p)$.
\end{proof}

\subsection*{Isolated singularities}
Suppose that $\Sing(X)$ is an isolated point $x$.
Choose a good resolution $\pi:Y\to X$, meaning that $Y$ is smooth,
$\pi$ is proper and an isomorphism away from $x$,
and $E=\pi^{-1}(x)_\red$ is a normal crossings divisor
with smooth components.
Then by \cite[4.8, 4.11]{duBois} we have a distinguished triangle
$$
0 \to \uOmega^p_X \to \R\pi_* \Omega^p_Y \oplus \Omega^p_{x}
\to \R\pi_* \uOmega^p_E \to 0.
$$
To rewrite this, let $E_1,\dots,E_t$ denote 
the (smooth) components of $E$, and define
\begin{equation}\label{Yn-point}
Y_n =\begin{cases} Y \amalg x_0 & n=0\\
\coprod_{i_1<\cdots<i_n}E_{i_1}\times_Y\cdots\times_Y E_{i_n} & n>0.
\end{cases}
\end{equation}
By \cite[4.10]{duBois}, the complex $\uOmega^p_E$ is
quasi-isomorphic to (the total complex of)
$$
\uOmega^p_{Y_1} \to \uOmega^p_{Y_2} \to \cdots.
$$
The maps in this complex are given by the usual alternating sum of
restriction maps, since the complex arises from a 
coskeletal hyperresolution of $E$. Generically writing $\pi:Y_n\to X$ 
for the canonical map from $Y_n$ to $X$, we have
\begin{equation} \label{E1}
\uOmega^p_X \simeq
\Tot \biggl(\R\pi_*\uOmega^p_{Y_0} \to \R\pi_*\uOmega^p_{Y_1} \to
  \R\pi_*\uOmega^p_{Y_2} \to \cdots\biggr).
\end{equation}

Now suppose that $\dim(X)=2$. Because 
$E_i \times_Y E_j \times_Y E_l = \emptyset$ for distinct $i,j,l$  
and $\Omega^p_{E_i \times_Y E_j}=0$ for $i\ne j$ and $p > 0$, \eqref{E1}
reduces to:
$
\uOmega^p_X \simeq \Tot\!\left(\R \pi_* \Omega^p_Y\!\to
    \bigoplus_i \R \pi_* \Omega^p_{E_i} \right)$
for $p>0$, and
\[
\uOmega^0_X \simeq
\Tot\biggl(\R\pi_*\cO_Y \oplus \cO_{x} \to
\bigoplus_i \R\pi_* \cO_{E_i} \to \bigoplus_{i<j} \R\pi_*\cO_{E_i \times_Y E_j}
\biggr).
\]
In other words, 
in the notation of \cite{Wahl},
\begin{equation}\label{eq:O(-E)}
\uOmega^0_X\simeq\R\pi_*\cO_Y(-E)\oplus\cO_x\quad\textrm{and}\quad
\uOmega^p_X \simeq \R \pi_* \left(\Omega^p_Y(\log E)(-E)\right),\quad p>0.
\end{equation}

\subsection*{Du Bois invariants}
Suppose for simplicity that $X=\Spec R$, where $R$ is a domain 
of finite type over $F$. For any $p \geq 0$, there is a map
from the $p$-th higher cotangent complex $\ctL^p_{X}$ 
(see \cite[3.5.4]{LodayHC92})
to the $p$-th du Bois complex $\uOmega^p_X$, obtained by composing the
isomorphism $H_0(\ctL^p_X) \cong \Omega^p_X$ and the natural map
$\Omega^p_X \to H^0(\uOmega^p_X)$.

\begin{defn}\label{def:CpX} 
Define the cochain complex $C^p_X$ of quasi-coherent
$\cO_X$-modules by
$$
C^p_{X} := \cone(\ctL^p_{X} \to \uOmega^p_X).
$$
That is, we have a triangle
$\ctL^p_X\to\uOmega^p_X\to C^p_X\to\ctL^p_X[1]$.
\end{defn}

In the language of \cite{nk}, the complex $C^p_X$ gives the homotopy
fiber $\cF_{HH}$ of the map from the Hochschild complex of $X$ to its 
$cdh$-fibrant replacement:
\begin{equation}\label{hcp=hfhh}
\bbH^i(C^p_X)=H^{i+1-p}(\cF_{HH}^{(p)}(X)).
\end{equation}
Note that the hypercohomology sheaves of $C^p_X$ are
coherent because  the K\"ahler differentials are taken over $F$.
Using  Lemma \ref{lem:duBois=cdh}, \cite[4.5.13]{LodayHC92}
and \cite[Lemma 3.4]{nk}, we conclude that: 
\begin{equation}\label{eq:FHC}
\bbH^q(C^p_X) =
\begin{cases}
\bbH^q(X, \uOmega^p) & \text{for $q \geq 1$} \\
\coker\left(\Omega^p_X \to \bbH^0(X, \uOmega^p)\right)
& \text{for $q = 0$} \\
\ker\left(\Omega^p_X \to \bbH^0(X,\uOmega^p)\right)
& \text{for $q =-1$} \\
D^{(p)}_{-1-q}(X) & \text{for $q \leq -2$,} \\
\end{cases}
\end{equation}
where $D^{(p)}_n$ denotes {\em Andr\'e-Quillen homology}. Recall that
$D^{(p)}_n(R) \cong HH^{(p)}_{p+n}(R)$, where 
$$
HH_* = \prod_{p \geq 0}
HH^{(p)}_*
$$ 
is the Hodge decomposition of Hochschild homology.

If $X$ has isolated singularities, then each of the hypercohomology modules
$\bbH^n(C^p_X)$ is of finite length. In this case we may define,
following and expanding on Steenbrink's definition \cite{Steenbrink},
the {\em generalized du Bois invariants} to be the numbers
\begin{equation}\label{eq:dBinvars}
b^{p,q} = b^{p,q}_X = \len\bbH^q(C^p_X), \quad \text{for $p \geq 0$ and
  $q \in \Z$.}
\end{equation}

\begin{subex} 
For $p=0$, we have $b^{0,q}=0$ if $q<0$. When $R$ is a domain,
   $b^{0,0}=\len_R(R^+/R)$,
where $R^+$ is the seminormalization of $R$, because
   $\ctL^0 = \cO_X$ and $H_\cdh^0(R,\cO) = R^+$ by \cite[2.5]{nk}.
If $q>0$ then \eqref{eq:O(-E)} yields $b^{0,q}=h^q(\cO_Y(-E))$.
\end{subex}

If, moreover, $X$ is locally a complete intersection, 
then $HH^{(p)}_n(R)=0$ for $n \gg 0$ (see \cite{FT}); hence it follows
from \eqref{eq:FHC} that $C^p_X$ is homologically bounded. 

\begin{defn} \label{def:chi} 
For a local complete intersection $X \in \SchF$ with only
isolated singularities, define $\chi^p(X)$ for $p \geq 0$ to be the Euler
characteristic of $C^p_X$:
$$
\chi^p(X) := 
\sum\nolimits_q (-1)^q b_X^{p,q}.
$$
\end{defn}

\begin{lem}\label{b-dR}
If $X = \Spec(R)$ for a ring $R$ that admits a non-negative grading
with $R_0 = k$, 
then
$\sum(-1)^p b_X^{p,q}=0$ for all $q$.
\end{lem}

\begin{proof}
The cases $q=-1$, $q=0$, $q>0$ follow from \eqref{eq:FHC} using the 
exact sequences 
\begin{align}  
0\to \nil(R) \to & \tors\Omega^1_R \to \tors\Omega^2_R \to
\tors\Omega^3_R \to\cdots \\ 
0\to (R^+/R) \to & \Omega^1_\cdh(R)/\Omega^1_R \to
\Omega^2_\cdh(R)/\Omega^2_R \to \cdots \\
0\to \Hcdh^n(R,\cO) & \map{d} \Hcdh^n(R,\Omega^1) \map{d}
    \Hcdh^n(R,\Omega^2) \to \cdots, \qquad n>0,
\end{align}
respectively, which are established in \cite[Example 3.9]{nk}.
For $q<-1$ it follows from Goodwillie's Theorem \cite[9.9.1]{WeibelHA94}.
\end{proof}

A key property of $\chi^p$ is its invariance under deformations of the sort
described in the following theorem.  In it, we write $X_s$ for the fiber
of $X$ over a point $s \in S$.

\begin{thm}\label{thm:chi}
Suppose $X \to S$ is a flat local complete intersection map of affine
varieties with $S$ smooth and such that the singular locus $X_{\sing}$
of $X$ is finite and \'etale over $S$.  Suppose in addition that one
can find a projective map $\pi: Y \to X$ which is an isomorphism away
from $X_\sing$, such that $Y$ is smooth and
such that the reduced, irreducible components $E_1, \dots, E_m$ of
$Y \times_X X_{\sing}$ are smooth over $S$ and
satisfy the property that each
$$
E_{i_1, \dots, i_t} :=
E_{i_1} \times_Y E_{i_2} \times_Y\cdots\times_Y E_{i_t} \to S
$$
is smooth 
($1\!\le\! i_1,\dots, i_t\!\le\! m$).
Then $\chi^p(X_s)$ is independent of the closed point $s$.

Suppose in addition that a finite group $G$ acts on both $X$ and $Y$
and that $\pi$ and $X \to S$ are equivariant, where we declare the
action of $G$ on $S$ to be trivial. Assume that $X/G \to S$ is a flat
local complete intersection such that $(X/G)_{\sing}$ is finite and
\'etale over $S$.
Then $\chi^p(X_s/G)$ is independent of the closed point $s\in S$.
\end{thm}

\begin{proof}
\addtocounter{equation}{-1}
\begin{subequations}

\noindent
In analogy with Definition \ref{def:CpX}, 
we use \eqref{E1} to define a relative version of $C^p$: 
$$
C^p_{X/S} :=
\Tot\left(
\ctL^p_{X/S} \to \R\pi_*\uOmega^p_{Y_0/S} \to
\R\pi_*\uOmega^p_{Y_1/S} \to \R\pi_*\uOmega^p_{Y_2/S} \to\cdots\right),
$$
where, as in \eqref{Yn-point}, 
\begin{equation}\label{Yn}
Y_n =\begin{cases} Y \amalg X_{\sing}& n=0\\
\coprod_{i_1<\cdots<i_n} E_{i_1,\dots,i_n}& n>0.
\end{cases}
\end{equation}
and $\ctL^p_{X/S}$ is the $p$-th cotangent complex for $X\to S$; the
map $\ctL^p_{X/S} \to \R\pi_*\Omega^p_{Y_0/S}$ is induced by
the composite of the natural maps
$\ctL^p_{X/S} \to \Omega^p_{X/S} \to \pi_* \Omega^p_{Y_0/S}$.

The complex $C^p_{X/S}$ is a complex of quasi-coherent $\cO_X$-modules
with only finitely many non-zero homology sheaves, each of which is coherent.
Moreover, each such homology sheaf is supported on the singular
locus of $X$, which maps finitely to $S$.
By restriction of scalars along the affine map $X \to S$, we may
therefore regard $C^p_{X/S}$ as a complex of quasi-coherent $\cO_S$-modules
whose homology is coherent.
As such, this complex determines a class $[C^p_{X/S}]$
in $G_0(S) = K_0(S)$. Explicitly, this class is
the alternating sum of these homology modules.

For any point $s \in S$,  let $j_s: s\to S$ 
be the induced map of schemes and let $j_s^*: K_0(S) \to K_0(s)\cong \Z$
be the pull-back map in $K$-theory. Note that for any $s$, the map $j_s^*$
sends the class of a locally free $\cO_S$-module to its rank.
Consequently, the map $j_s^*:K_0(S)\to\Z$ 
does not depend on the choice of $s \in S$.
We now prove that for any closed point $s\in S$:
\begin{equation}\label{eq:CpX/S}
j_s^*[C^p_{X/S}] = [C^p_{X_s/s}].
\end{equation}
Since the class $[C^p_{X_s/s}]$ in $K_0(s) = \Z$ is $\chi^p(X_s)$
when $s\in S$ is a closed point, 
this will prove the first assertion of the Theorem.

Note first of all that if $\cF^\bu$ is any complex of quasi-coherent
$\cO_S$-modules with bounded, coherent homology, then $j_s^*[\cF^\bu] =
[\bL j_s^* \cF^\bu]$, where $\bL j^*_s$ denotes the left derived functor
associated to $j_s^*$. For any $n$, let $\tilde{\pi}: Y_n \to S$
be the structure map, which we are supposing to be smooth and hence flat.
Thus $\tilde{\pi}$ and $j_s$ are Tor-independent over $S$.
Consider the pullback diagram
\[
\xymatrix{(Y_n)_s\ar[d]_q\ar[r]^{\alpha_s}& Y_n \ar[d]^{\tilde{\pi}}\\
            s\ar[r]^{j_s}& S.}
\]
By \cite[IV.3.1]{SGA6}, we have
$
\bL j_s^* \R \tilde{\pi}_* \Omega^p_{Y_n/S} \simeq
\R q_* \bL \alpha_s^* \Omega^p_{Y_n/S}.
$
Since $Y_n/S$ is smooth, 
$\Omega^p_{Y_n/S}$ is locally free and we have
$$
\bL\alpha_s^*\Omega^p_{Y_n/S} =
\alpha_s^* \Omega^p_{Y_n/S} \cong \Omega^p_{(Y_n)_s /s}.
$$
Hence
\begin{equation}\label{eq:bchange}
\bL j_s^* \R \tilde{\pi}_* \Omega^p_{Y_n/S} \simeq
\R q_* \Omega^p_{(Y_n)_s /s}.
\end{equation}
Similarly, it is a standard property of the cotangent complex that
$$
j_s^* \ctL^p_{X/S} \simeq \bL j_s^* \ctL^p_{X/S} \simeq \ctL^p_{X_s/s}.
$$
Combining these, we get the formula 
$$
j_s^*[C^p_{X/S}]
=
\left[\cdots \to
\ctL^p_{X_s/s} \to \R q_* \Omega^p_{(Y_0)_s /s} \to
\R q_* \Omega^p_{(Y_1)_s /s} \to \cdots \right].
$$
Finally, if $s$ is a closed point then by \eqref{E1} we have
$$
\uOmega^p_{X_s} \simeq
\left(
\R q_* \Omega^p_{(Y_0)_s /s} \to
\R q_* \Omega^p_{(Y_1)_s /s} \to \cdots \right)
$$
and hence the formula $j_s^*[C^p_{X/S}] = [C^p_{X_s/s}]$
of \eqref{eq:CpX/S}, proving the first assertion.

Suppose now that a finite group $G$ acts on $X$ and $Y$ as
in the statement of the Theorem. Let $Y_n$ be as in \eqref{Yn} above;
then $G$ acts on $Y_n\to S$ and hence on $\uOmega^p_{Y_n/S}$ and
$\R \tilde{\pi}_*\uOmega^p_{Y_n/S}$ for all $n$. For each $s\in S$,
the group $G$ acts also on $\uOmega^p_{(Y_n)_s}$.

Since $G$ is a finite group and we are in characteristic $0$, taking
$G$-invariants is exact. This implies the key property we will need,
proven in \cite[5.12]{duBois}, namely that
$$
\uOmega^p_{(Y_n)_s/G} \simeq
(\uOmega^p_{(Y_n)_s})^G \simeq
(\Omega^p_{(Y_n)_s})^G.
$$
Since taking $G$-invariants also commutes with $\R q_*$,
this property implies that
\begin{equation}\label{eq:OmegaY/G}
\R q_*\bigl(\uOmega^p_{(Y_n)_s/G}\bigr)
\simeq
\R q_*\bigl((\Omega^p_{(Y_n)_s})^G\bigr)
\simeq
(\R q_*\Omega^p_{(Y_n)_s})^G.
\end{equation}
Define the analogue $D^p_{X/S}$ of $C^p_{X/S}$ by
$$
D^p_{X/S} =
\left(\ctL^p_{(X/G)/S} \to (\R\pi_*\Omega^p_{Y_0/S})^G  \to
    (\R\pi_*\Omega^p_{Y_1/S})^G \to \cdots \right).
$$
Now taking $G$-invariants commutes with $\bL j_s^*$.
Using \eqref{eq:bchange} and \eqref{eq:OmegaY/G}, we have
$$
\bL j_s^* ((\R \tilde{\pi}_* \Omega^p_{Y_i/S})^G)
\simeq
(\bL j_s^* (\R \tilde{\pi}_* \Omega^p_{Y_i/S}))^G
\simeq
(\R q_* \Omega^p_{(Y_i)_s})^G
\simeq
\R q_* (\uOmega^p_{(Y_i)_s/G)}).
$$
Finally, observe that a similar argument as that used to prove
\eqref{E1} shows that
$$
\uOmega^p_{X_s/G} \simeq
\left(
\R q_* \uOmega^p_{((Y_0)_s/G)} \to
\R q_* \uOmega^p_{((Y_1)_s/G)} \to \cdots \right).
$$
Indeed, $X_s/G, Y_s/G$, and the $(E_i)_s/G$
satisfy the same hypotheses as do $X$, $Y$, and the $E_i$, except for
smoothness, so that the results in  \cite[4.8, 4.10, 4.11]{duBois} apply.
It follows that
$$
j_s^* [D^p_{X/S}] \simeq [C^p_{(X_s/G)}].
$$
Since the class of $[C^p_{X_s/G}]$ in $K_0(s)=\Z$ is $\chi^p(X_s/G)$,
it is independent of $s$. 
\end{subequations}
\end{proof}

\section{Isolated (hyper)surface singularities.}\label{sec:Wahl}

In this section we consider the du Bois invariants of a two-dimensional
isolated hypersurface singularity $X$.
That is, $X = \Spec R$ where $R = F[x,y,z]/(f(x,y,z))$ and
$\Omega^3_{R/F} \cong R/(\fx, \fy, \fz)$ is supported at the origin (\ie
the unique singular  point $x_0$ is defined by the maximal ideal $(x,y,z)$).
The analytic analogues of some of our results are due to Steenbrink
and may be found in Wahl's paper \cite{Wahl}.

We will need the following well known calculation of
$\Omega^p_R=\Omega^p_{R/F}$.
Recall that the {\em Tjurina number} $\tau$ is: 
\[
\tau = \len_R \left(R/{\textstyle \left(\fx, \fy, \fz\right)}\right).
\]

\begin{lem}\label{lem:lengths}
Let $X = \Spec R$ be a 2-dimensional isolated hypersurface singularity.
Then each of the following $R$-modules has length equal to
$\tau$:
$$
\Omega^3_R\cong \Ext^1_R(\Omega^1_R,R)\cong\Ext^2_R(\Omega^2_R, R),
\quad \tors(\Omega^2_R)\cong\Ext^1_R(\Omega^2_R, R).
$$
\end{lem}

\begin{proof}
Write $R=P/f$, where $P=F[x,y,z]$, and
consider the complex $\cK$ of free $R$-modules,
whose boundary maps are induced by exterior multiplication with $df$,
indexed with $R$ in degree~0: 
\[
\cK:\quad 0\to R\overset{\land
  df}\to\Omega^1_P\otimes_PR\overset{\land
  df}\to\Omega^2_P\otimes_PR\overset{\land
  df}\to\Omega^3_P\otimes_PR\to 0.
\]
By \cite[p.\ 326]{Mi94}, the $n$-th cohomology of the complex $\cK$ is
the torsion submodule of $\Omega^n_R$. In the isolated singularity
case considered here, it follows from Lebelt's results \cite{Lebelt}
(see also \cite[Prop.\ 1]{Mi95}) that $\Omega^n_R$ is a torsionfree
module for $n\le 1$. In particular, we have free resolutions:
\begin{gather}
0\to R\overset{\land df}\to\Omega^1_P\otimes_PR\to\Omega^1_R\to 0
\nonumber \\
0\to R\overset{\land df}\to\Omega^1_P\otimes_PR\overset{\land
  df}\to\Omega^2_P\otimes_PR\to \Omega^2_R\to 0 \nonumber
\end{gather}
Moreover the perfect pairing
$\Omega^p_P\otimes_P\Omega^{3-p}_P\to\Omega^3_P\cong P$
induces a perfect pairing $\cK^p\otimes_R \cK^{3-p}\to \cK^3\cong R$.
 From this we get an isomorphism of complexes
$\Hom_R(\cK,R)[-3]\cong \cK$. It follows that
\begin{gather}
\Ext^1_R(\Omega^1_R, R)=H^3(\cK)=\Omega^3_R \cong
R/{\textstyle \left(\fx, \fy, \fz\right)}, \nonumber \\
\Ext^1_R(\Omega^2_R,R)=H^2(\cK)=\tors\Omega^2_R. \nonumber
\end{gather}
By definition, the length of the first of these modules is $\tau$;
by \cite[Thm.\ 3]{Mi95}, the second module also has
length $\tau$.
\end{proof}

Recall the definition of the (generalized) du Bois invariants $b^{p,q}$ from
\eqref{eq:dBinvars}.

\begin{prop}\label{prop:b-bounds}
Let $X = \Spec R$ be a 2-dimensional isolated hypersurface singularity.
Then the following hold:
\begin{enumerate}
\item[(a)] $b^{p,q} = 0$ unless $p+q \in \{1,2\}$.
\item[(b)] $b^{1-q,q}=b^{2-q,q}=\tau$ for all $q<0$.
\item[(c)] $b^{0,2} = 0$, \textrm{ and }$b^{0,1}=-\chi^0$.
\end{enumerate}
\end{prop}

\begin{proof}
To prove (a), note that 
for $q>0$, it is a particular case of a general statement for
isolated singularities 
proved by Steenbrink in \cite[Thm.\ 1]{Steenbrink}, since
$b^{p,q}$ is the length of $\bbH^q(X,\uOmega^p_X)$ by \eqref{eq:FHC}.
In our case Steenbrink's result is immediate from Grauert-Riemenschneider
vanishing \cite[Satz 2.3]{GrauRiem} and from the fact, proved in
\cite[Prop.\ 2.6]{chw}, that for an affine surface $X$,
\begin{equation}\label{h2vanish}
H^2_\cdh(X,\Omega^p)=0\qquad (p\ge 0).
\end{equation}
If $q = 0$ and $p > 2$, (a) holds since then $a^*\Omega^p =0$.
If $q = p = 0$, it holds since $R$ is normal, hence seminormal.
For $q = -1$, (a) holds because  $\Omega^p_R = 0$ for $p > 3$ and 
$R$ and $\Omega^1_R$ are torsionfree; see
\cite[Lemma 5.6 and Remark 5.6.1]{nk}. 
For $q\le -2$,
we have
\begin{equation} \label{Michlertor}
H_q(C^p_X)=D^{(p)}_{-1-q}(R)=HH_{p-q-1}^{(p)}(R)=\tors(\Omega^{p+q+1}_R)
\end{equation}
which is zero unless $p+q\in\{1,2\}$, by a result of Michler \cite{Mi94}.

Assertion (b)
follows from \eqref{Michlertor} and the fact that the kernel of
  $\Omega^n_R \to H^0_\cdh(X, \Omega^n)$ is $\tors(\Omega^n_R)$
  (see \cite[Lemma 5.6 and Remark 5.6.1]{nk}), using \cite[Thm.\ 3]{Mi95} (see Lemma
  \ref{lem:lengths}). 

For assertion (c),
the vanishing of $b^{0,2}$ is a particular case of \eqref{h2vanish}. 
The other assertion 
follows from part (a) and the definition (see \ref{def:chi}) of $\chi^0$.
\end{proof}

\begin{prop} \label{prop:ROS}
Let $X = \Spec R$ be a 2-dimensional isolated hypersurface singularity.
Further let $\pi: Y\to X$ be a good resolution, $E$ the exceptional divisor,
$E_1,\dots,E_r$ its reduced irreducible components, $g_i$ the genus of $E_i$,
and $l$ the number of loops in the incidence graph. Put $g=\sum_ig_i$
and $p_g=\len_R H^1(Y,\cO_Y)$.

\begin{enumerate}
\item[(a)] The map $H^n_{\cdh}(X,\cO)\to H^n_\cdh(Y,\cO)=H^n(Y,\cO)$ is
an isomorphism for $n\ne1$, and an injection for $n=1$. We have
\[
b^{0,1}= p_g-g-l.
\]
In particular $H^1_\cdh(X,\cO)\to H^1(Y,\cO)$ is an isomorphism $\iff$ $g=l=0$.

\item[(b)]
$H^n_\cdh(X, \Omega^2) \cong H^n(Y, \Omega^2)$ for $n\!\ge\!0$.
In particular, $H^n_\cdh(X,\Omega^2)=0$ for $n\!\ge\!1$.

\item[(c)] 
$\Ext^i_R( H^0(Y,\Omega^2),R) \cong H^i(Y,\cO_Y).$
In particular, $Ext^2_R(H^0(Y,\Omega^2),R)=0$.

\item[(d)] $b^{1,0}\le\tau$. 

\item[(e)] 
$b^{2,0} = \tau - p_g$, and $\chi^2=-p_g$.
\end{enumerate}
\end{prop}

\begin{proof}
To prove (a), observe that $R$ is normal and $Y\to X$ is projective,
so that $R=H^0_\cdh(X,\cO)=H^0(Y,\cO)$ by Zariski's Main Theorem
(and  \cite[Proposition 2.5]{nk}). Since $Y\to X$ has fibers of
dimension at most $1$, and $X$ is affine,
\begin{equation}\label{fibers}
H^2(Y,\cF)= H^0(X,\R^2\pi_*\cF) = 0
\end{equation}
for all coherent sheaves $\cF$. In particular,
$H^2(Y,\cO)=0$. Similarly, $H^2_\cdh(X,\cO)=0$ by
\cite[Theorem 6.1]{chsw}. 
Since $\Sing X=\{x_0\}$, we have a blowup square 
\begin{equation}\label{goodres}
\xymatrix{E\ar[d]\ar[r]&Y\ar[d]\\ x_0\ar[r]&X} 
\end{equation}
 From the Mayer-Vietoris sequence associated to this square,
we extract the short exact sequence
\[
0\to H^1_\cdh(X,\cO)\to H^1(Y,\cO)\to H^1_\cdh(E,\cO)\to 0.
\]
Hence $b^{0,1}=\len_R H^1(Y,\cO)-\len_RH^1_\cdh(E,\cO)$. Applying
descent to the cover $\coprod_iE_i\to E$, we obtain $\len_R
H^1_\cdh(E,\cO)=l+g$.

For (b), the isomorphisms $H^n_\cdh(X,\Omega^2)\cong H^n(Y,\Omega^2)$
follow from the Mayer-Vietoris sequence associated to the square
\eqref{goodres}. By Grauert-Riemenschneider vanishing
\cite[Satz 2.3]{GrauRiem}, $\R\pi_*\Omega^2_Y \simeq \pi_*\Omega^2_Y$,
so $H^n(Y,\Omega^2)=H^0(X,R^n\pi_*\Omega^2_Y)$ vanishes for $n>0$
because $X$ is affine.

To prove (c), recall that $\omega_X\cong\cO_X[2]$ because
$X$ is an affine hypersurface. 
For any bounded complex of quasi-coherent sheaves $\cF^\bu$ on $Y$,
Grothendieck-Serre duality gives a quasi-isomorphism:
$$
\R \pi_* \R \Hom_Y(\cF^\bu, \Omega^2_Y)  \simeq
\R \Hom_X(\R \pi_* \cF^\bu, \cO_X)
$$
Taking $\cF^\bu = \Omega^p_Y$ and using the duality pairing on $Y$,
$$
\R \Hom_Y(\Omega^p_Y, \Omega^2_Y) \simeq
\Hom_Y(\Omega^p_Y, \Omega^2_Y) \cong \Omega^{2-p}_Y,
$$
we get a spectral sequence
\begin{equation}\label{gsss}
\Ext^i_R(H^j(Y,\Omega^p),R) \Rightarrow H^{i-j}(Y, \Omega^{2-p}).
\end{equation}
Taking $p = 2$
and using Grauert-Riemenschneider vanishing \cite[Satz 2.3]{GrauRiem},
which gives $H^j(Y,\Omega^2)=0$ for $j>0$, 
we obtain the conclusion of (c):
$$
\Ext^i_R( H^0(Y,\Omega^2),R) \cong H^i(Y,\cO_Y).
$$
In particular, by \eqref{fibers},
$
\Ext^2_R( H^0(Y,\Omega^2),R) = 0.
$

To prove (d), recall that $b^{1,0}$ is the length of the
$R$-module $L=\bbH^0(C^1_X)$. 
Since $b^{1,-1}=0$ by Proposition
\ref{prop:b-bounds}, it follows from \eqref{eq:FHC} that
we have an exact sequence
\begin{equation}\label{moduleL}
0 \to \Omega^1_R \to H^0_\cdh(X, \Omega^1) \to L\to 0.
\end{equation}
From \eqref{moduleL} we get the exact sequence
\begin{equation}\label{extseq}
\Ext^1_R(\Omega^1_R, R) \to
 \Ext_R^2(L, R)  \to \Ext^2_R(H^0_\cdh(X, \Omega^1), R).
\end{equation}
 From the spectral sequence \eqref{gsss} with $p=1$, we have an
exact sequence 
$$
\Hom_R(H^1(Y,\Omega^1),R) \map{d_2} \Ext^2_R(H^0(Y,\Omega^1),R)
\to H^2(Y,\Omega^1).
$$
Since the $R$-module $H^1(Y,\Omega^1)$ is supported at $x_0$,
$\Hom_R(H^1(Y,\Omega^1),R)=0$. The right side also vanishes,
by \eqref{fibers}, so we get $\Ext^2_R(H^0(Y,\Omega^1),R) = 0$.

By part (a), the map $H^0_\cdh(X,\Omega^1) \to H^0(Y,\Omega^1)$ is
injective, so the map
$$
\Ext^2_R(H^0(Y,\Omega^1), R) \to \Ext^2_R(H^0_\cdh(X,\Omega^1), R)
$$
is surjective and hence
$$
\Ext^2_R(H^0_\cdh(X, \Omega^1), R) = 0.
$$
 From \eqref{extseq} we get that $\Ext^1_R(\Omega^1_R,R)\to\Ext_R^2(L, R)$
is surjective and hence
\begin{align*}
b^{1,0}=\len_R(L)=&~\len_R(\Ext_R^2(L, R))\\
  \leq& ~\len_R(\Ext^1_R(\Omega^1_R,R))\\
  =& ~\tau, \text{ by Lemma \ref{lem:lengths}}.
\end{align*}

To prove (e), define finite length $R$-modules $N$ and $M$ so that
\begin{equation}\label{seqmn}
0 \to N \to \Omega^2_R \to H^0(Y, \Omega^2) \to M \to 0
\end{equation}
is exact. By definition \eqref{eq:dBinvars} and the fact that
$R$ is Gorenstein, we get
\begin{equation}\label{b20=ext2}
b^{2,0} = \len_R(M)=\len_R(\Ext^2_R(M,R)).
\end{equation}
Because $N$ has finite length, $\Ext^i(N,R)=0$ for $i<2$ and hence
there are isomorphisms 
\[
\Ext^i(\Omega^2_R/N,R)\overset\simeq\to \Ext^{i}(\Omega^2_R,R)\qquad (i<2).
\]
Using this together with part (c) and \eqref{seqmn},
we get an exact sequence
$$
0 \to H^1(Y, \cO_Y) \to \Ext^1_R(\Omega^2_R, R) \to \Ext^2_R(M, R) \to
0.
$$
Using this sequence, and taking into account  
Lemma \ref{lem:lengths} and \eqref{b20=ext2}, we get
$$
b^{2,0} = \tau - \len_RH^1(Y,\cO) = \tau - p_g.
$$
By \ref{prop:b-bounds}(a,b), this yields $\chi^2=b^{2,0}-\tau=-p_g$.
\end{proof}

\section{Wahl's example.}\label{sec:CE}

Using the general results of the preceding sections, we can now prove:

\begin{thm}\label{thm:dBcompute}
Let $F$ be a field of characteristic $0$ and
$$
R = F[x,y,z]/(z^2 + y^3 + x^{10} + ax^7 y),
$$
for any nonzero $a \in F$. Then $b^{0,1}=1$ and $b^{1,1}=0$. That is,
\begin{enumerate}
\item[(a)] $H^1_\cdh(R,\cO)\cong F$ and
\item[(b)] $H^1_\cdh(R,\Omega^1_{/F}) = 0$.
\end{enumerate}
\noindent 
In particular, if $F$ is an algebraic field extension of $\Q$, then
$R$ gives a negative answer to Bass' question for $n=0$:
\[
K_0(R)=K_0(R[t]) \textrm{ but } K_0(R[t_1,t_2])\cong K_0(R)\oplus stF[s,t].
\]
\end{thm}

\begin{subremark}
The $cdh$ cohomology groups in question may also be computed using an
explicit description of a resolution of singularities, together with
the self-intersection numbers of the exceptional components. For the
surface in Theorem \ref{thm:dBcompute} for all values of $a$
(including $0$), 
the resolution data
was checked for us by Liz Sell, and is displayed in Figure 1.
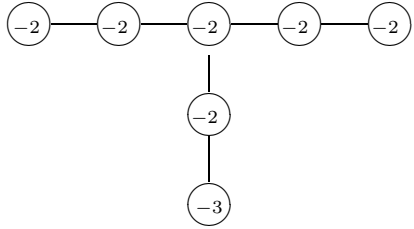
\begin{figure}[h]
\begin{center}
\begin{picture}(200,100)
\setlength{\unitlength}{0.60cm} 
\put(0.0,4){\circle{1}}
\put(-.35,3.8){$\scriptstyle-2$}
\put(2.0,4){\circle{1}}
\put(1.6,3.8){$\scriptstyle-2$}
\put(4.0,4){\circle{1}}
\put(3.6,3.8){$\scriptstyle-2$}
\put(6.0,4){\circle{1}}
\put(5.6,3.8){$\scriptstyle-2$}
\put(8.0,4){\circle{1}}
\put(7.6,3.8){$\scriptstyle-2$}
\put(4.0,2){\circle{1}}
\put(3.6,1.8){$\scriptstyle-2$}
\put(4.0,0){\circle{1}}
\put(3.7,-.2){$\scriptstyle-3$}
\put(.5,4){\line(1,0){1.0}}
\put(2.5,4){\line(1,0){1}}
\put(4.5,4){\line(1,0){1}}
\put(6.5,4){\line(1,0){1}}
\put(4,2.5){\line(0,1){.8}}
\put(4,0.5){\line(0,1){1}}
\end{picture}
\caption{The Resolution graph for $z^2+y^3+x^{10}+ax^7y$}
\end{center}
\end{figure}
\end{subremark}

The proof we shall give here will be a straightforward
application of the invariance of $\chi^p$ (Theorem \ref{thm:chi}),
applied to the specific example:
\begin{equation}\label{eq:deform}
X=\Spec F[x,y,z,t]/(z^2 +  y^3 + x^{10} + tx^7y).
\end{equation}
Consider the map $X \to S = \Spec F[t]$ induced by the
obvious inclusion of rings, and write $X_s$ for the fiber over $s \in S$.
When $s$ is the point $t=a$ we have $X_s=\Spec(R)$ for the ring $R$
in Theorem \ref{thm:dBcompute}.

\begin{prop}\label{prop:Wahl}
Let $X$ be the affine variety of \eqref{eq:deform}.
Then  the integer $\chi^p(X_s)$ is independent of the choice
of closed point $s\in S$. 
\end{prop}

\begin{proof} Since the value of $\chi^p$ does not change upon
passing to a finite extension,  we may
  assume that $F$ contains a primitive $30$-th root of unity.
Put
$$
\tilde{X} = \Spec F[u,v,w,t]/(u^{30} + v^{30} +
  w^{30} + tu^{21}v^{10}).
$$
Let $G = \mu_{3} \times \mu_{10} \times \mu_{15}$ act on $\tilde{X}$
by scalar multiplication on the variables $x,y,z$ 
so that the assignment $x = u^3$, $y = v^{10}$ and $z = w^{15}$
identifies $X$ with $\tilde{X}/G$.

The map $X\to S$ is a flat
local complete intersection whose singular locus is defined by
$x=y=z=0$ and hence maps isomorphically onto $S$. 
The singular locus of $\tilde{X}$ is defined by $u=v=w=0$
and hence also maps isomorphically onto $S$. 
Let $\tilde{Y}$ be the blowup of $\tilde{X}$
along its singular locus. Then
$$
\tilde{Y} = \Proj \left(\frac{F[t, u, v, w, A, B, C]}
{(A^{30} + B^{30} + C^{30} + tuB^{10} A^{20}, uB-vA, uC-wA, vC-wB)}\right),
$$
where $t, u, v, w$ have degree $0$ and $A,B,C$ have degree $1$.
Direct calculations show that $\tilde{Y} \to S$ is smooth and
the fiber of $\tilde{Y} \to \tilde{X}$ over
$\tilde{X}_{\sing}$ is
$$
\tilde{E} = \Proj F[t,A, B, C]/(A^{30} + B^{30} + C^{30})
\cong S \times E_0
$$
where $E_0$ is a smooth curve.  We see that all the hypotheses of
Theorem \ref{thm:chi} are satisfied.
\end{proof}

Since $X_0$ is quasi-homogeneous,  its du Bois invariants may be
computed, as shown in the following example.
These calculations and the above
results lead to the proof of Theorem \ref{thm:dBcompute} below.

\begin{ex}\label{ex:Wahl4.4}
The surface $X_0=\Spec F[x,y,z]/(z^2+y^3+x^{10})$
is discussed by Wahl in \cite[4.4]{Wahl}.
Elementary calculations, described in \cite[4.3]{Wahl}, give
that $\tau=1\cdot2\cdot9=18$, $g=0$ and
\[
p_g=\dim \left(F[x,y,z]/{\textstyle (\fx, \fy, \fz)}\right)_{\le2}=1
\]
where $f=z^2+y^3+x^{10}$. Moreover, 
as with any isolated normal surface
singularity defined by a non-negatively graded ring,
we have $l=0$ by \cite[Theorem 2.3.1]{OW}. (Or, one may see $l=0$ from
the graph of Figure ~1.) 
Using Lemma \ref{b-dR} and Proposition \ref{prop:ROS}(a,e), this yields
\[
b^{1,1} = b^{0,1} = p_g-g-l=1, \quad 
b^{1,0} = b^{2,0} = \tau-p_g=17.
\]
By Proposition \ref{prop:b-bounds}(a), 
$\chi^0=-b^{0,1}=-1$, $\chi^1=b^{1,0}-b^{1,1}=16$, $\chi^2=-1$.
\end{ex}

\begin{proof}[Proof of Theorem \ref{thm:dBcompute}]
By Theorem \ref{thm:chi}, $\chi^p(X_s)$ does not depend on
$s$ and we write $\chi^p = \chi^p_s$. By Proposition \ref{prop:b-bounds}(c),
$b^{0,1}=-\chi^0$ is also independent of $s$. For the choice $s=0$,
we have $b_0^{0,1} = 1$ by \cite[4.4]{Wahl} (see Example
\ref{ex:Wahl4.4}).
This proves assertion (a).
To compute $b^{1,1}$ when $a\ne0$, 
we use the calculation of $\tau(X_a)$ given in \cite[4.4]{Wahl}:
\begin{equation}\label{tauW}
\tau(X_a)=\begin{cases}18 & a=0\\ 16& a\ne 0. \end{cases}
\end{equation}
By Proposition \ref{prop:ROS}(d)
\begin{equation}\label{b1016}
b^{1,0}(X_a)\le \tau(X_a) = 16 \quad \text{for all $a \ne 0$.}
\end{equation}

\noindent
By the invariance of $\chi^1$ (see Proposition \ref{prop:Wahl}),
Example \ref{ex:Wahl4.4} and \eqref{b1016}, we have
\begin{align*}
16 = \chi^1 &= b^{1,0}(X_a) - b^{1,1}(X_a) \\
             &\leq 16 - b^{1,1}(X_a) 
\end{align*}
for any $a \ne 0$, and hence
$0 = b^{1,1}(X_a) = \dim_F H^1_\cdh(X_a, \Omega^1_{}).$

The final assertion follows from Theorem \ref{mainnk}.
\end{proof}

\begin{rem} We conclude with a few remarks.

\begin{enumerate}
\item[(a)]
In \eqref{tauW} of the proof, 
we refer to the calculation of the Tjurina numbers $\tau$ 
stated by Wahl in \cite[4.4]{Wahl}. These can be
checked directly using the Tjurina function of the {\sc Singular}
library {\tt sing.lib} (\cite{GPS05}, \cite{GM}).

\item[(b)] Steenbrink uses analytic methods to define an invariant
$\alpha$ and proves that $b^{1,1}=p_g-g-l-\alpha$; see
\cite[(1.9.1)]{Wahl}. Comparing with Proposition
\ref{prop:b-bounds}(a), and using GAGA, 
we see that $\alpha = b^{0,1}-b^{1,1}$.
It is this invariant that is computed by Wahl in \cite[4.4]{Wahl}.

\item[(c)] If $R_F=F[x,y,z]/(z^2+y^3+x^{10})$ and $F$ is not algebraic
over $\Q$, then $NK_0(R_F)$ is nonzero. 
Indeed, $NK_0(R_F) \cong \Omega^1_{F/\Q}\oo_F tF[t].$ This follows from
\cite[(7.4)]{nk}, which says that
\[ 
NK_0(R_F)\cong NK_0(R_{\Q})\oo_{\Q}F \oplus 
            NK_{-1}(R_{\Q})\oo_{\Q}\Omega^1_{F/\Q},
\]
since $NK_0(R_{\Q})=0$ and $NK_{-1}(R_F) \cong tF[t]$ 
by Theorems \ref{mainnk}(b) and \ref{thm:dBcompute}(a).
\end{enumerate}
\end{rem}

\subsection*{Acknowledgements}
The authors would like to thank M.~Schlichting, whose contributions
go beyond the collaboration \cite{chsw}. We would also like to thank
W.~Vasconcelos and L.~Avramov for useful discussions,
and E.~Sell and J.~Wahl for their help in
checking our resolution of singularities.


\def\cprime{$'$}

\end{document}